\documentclass{article}
\usepackage{amsmath, amssymb, amsthm, enumitem}

\newtheorem{thm}{Theorem}[section]
\newtheorem{cor}[thm]{Corollary}
\newtheorem{prop}[thm]{Proposition}
\newtheorem{lem}[thm]{Lemma}

\theoremstyle{definition}
\newtheorem{defn}[thm]{Definition}
\newtheorem{exmp}[thm]{Example}

\title{Triangle-free quantum graphs}
\author{Nik Weaver}
\date{September 2025}

\begin{document}

\maketitle

\begin{abstract}
We introduce notions of being ``triangle-free'' and ``strongly triangle-free'' for operator systems in $M_n(\mathbb{C})$ considered as quantum graphs. Several examples and non-examples are discussed. We provide a complete characterization of strongly triangle-free operator systems.
\end{abstract}

\section{Introduction}

An {\it operator system} in finite dimensions is a linear subspace $\mathcal{V}$ of the complex matrix algebra $M_n = M_n(\mathbb{C})$ which contains the identity matrix ($I_n \in \mathcal{V}$) and is stable under Hermitian transpose ($A \in \mathcal{V} \to A^* \in \mathcal{V}$). In \cite{DSW}, motivated by ideas from quantum error correction, these objects were identified as a quantum analog of finite simple graphs. A crude intuition, not to be taken too seriously, could be that the operator system $\mathcal{V}$ represents the edge structure and the unit vectors in $\mathbb{C}^n$ play the role of vertices, with two unit vectors $v,w \in \mathbb{C}^n$ being ``adjacent'' if $\langle Av,w\rangle \neq 0$ for some $A \in \mathcal{V}$.

Operator systems in $M_n$ turn out to be a special case of a more general notion of a ``quantum graph'' over an arbitrary von Neumann algebra $\mathcal{M}$ \cite{daws, MRV, weavera, weaverb}, in which the $\mathcal{M} = l^\infty(X)$ case effectively replicates classical graphs with vertex set $X$ and $\mathcal{M} = M_n$ reduces to the finite state, purely quantum setting of \cite{DSW}.

Despite substantial recent interest in general quantum graphs, the elementary setting of \cite{DSW} --- the purely quantum analog of finite simple graphs --- remains relatively unexplored. We do have good notions of ``clique'' and ``anticlique'', and, for instance, there is a quantum analog of Ramsey's theorem which states that if $n$ is sufficiently large then every operator system in $M_n$ has either a $k$-clique or a $k$-anticlique \cite{weaverx}. We also know what it means for an operator system to be ``connected'': there should be no nontrivial projections in its commutant \cite{CDS}. But I am not even sure what the right notion of a ``quantum tree'' would be. One possibility is ``minimal connected operator system'', but is this a fruitful definition?

At any rate, it seems that there is still an opportunity to identify quantum analogs of standard notions from classical graph theory in the elementary matrix algebra setting. In this paper I will investigate quantum analogs of triangle-free graphs. Here are the basic definitions.

\begin{defn}
A {\it $k$-anticlique} for an operator system $\mathcal{V} \subseteq M_n$ is a rank $k$ projection (= Hermitian idempotent) $P \in M_n$ which satisfies $P\mathcal{V}P = \mathbb{C}\cdot P$, i.e., for which the compression of every operator in $\mathcal{V}$ to ${\rm ran}(P)$ is a scalar multiple of the identity operator on ${\rm ran}(P)$. Equivalently, ${\rm dim}(P\mathcal{V}P) = 1$. A {\it $k$-clique} is a rank $k$ projection $P$ which satisfies $P\mathcal{V}P \cong M_k$, i.e., for which every linear operator on ${\rm ran}(P)$ is the compression of something in $\mathcal{V}$. Equivalently, ${\rm dim}(P\mathcal{V}P) = k^2$.
\end{defn}

Classically, the notion of a ``triangle-free'' graph can be formulated in two trivially equivalent ways: (1) there is no triple of vertices among which all edges are present, or (2) among any triple of vertices at least one edge is absent. That is, we can say either that there is no 3-clique or that the induced subgraph on every triple of vertices contains a 2-anticlique. This leads to two distinct versions of an operator system being ``triangle-free''.

\begin{defn}
An operator system is {\it triangle-free} if it has no 3-clique. It is {\it strongly triangle-free} if every rank 3 projection dominates a 2-anticlique.
\end{defn}

The first version is already very strong. Generally speaking, it is easy to find cliques and hard to find anticliques in operator systems \cite{weavery}. (There is no simple duality between cliques and anticliques as there is in the classical setting.) So operator systems that do not have even a single 3-clique should be rare. However, they do exist; see Section 2 below.

The second version is even stronger: if every rank 3 projection dominates a 2-anticlique then there certainly cannot be any 3-cliques. This is because any 2-anticlique $P$ for $\mathcal{V}$ that is dominated by a rank 3 projection $Q$ (i.e., $P \leq Q$) will also be a 2-anticlique for the compression $Q\mathcal{V}Q$. So this compression cannot be isomorphic to $M_3$, because the latter obviously has no 2-anticliques. In fact, we can strengthen this statement a little; see Proposition \ref{plustwo} below.

We could also define an operator system $\mathcal{V}$ to be {\it edge-free} if it has no 2-cliques, and {\it strongly edge-free} if every rank $2$ projection is a 2-anticlique. But in fact these properties are simply equivalent to the conditions ${\rm dim}(\mathcal{V}) \leq 3$ and ${\rm dim}(\mathcal{V}) = 1$. The first is a consequence of Theorem 3.3 of \cite{weaverx} (plus the trivial fact that $\mathcal{V}$ cannot have a 2-clique if its dimension is at most 3), and the second is a restatement of the easy fact that if an operator $A \in M_n$ is not a scalar multiple of the identity then there exists a rank 2 projection $P \in M_n$ such that $PAP$ is not a scalar multiple of $P$.

In contrast to this result, there exist strongly triangle-free operator systems (and therefore also triangle-free operator systems) of arbitrarily large dimension; see Proposition \ref{1ntf} below.

\section{Examples}

\subsection{Small operator systems}

Let $(e_i)$ be the standard orthonormal basis of $\mathbb{C}^n$. We recall the following simple fact.

\begin{prop}\label{biganti}
(\cite{weavery}, Proposition 2.7)
Let $A \in M_n$ be Hermitian. Then ${\rm span}(I_n,A)$ has an $\lceil \frac{n}{2}\rceil$-anticlique.
\end{prop}

The proof is easy. If $n = 2m + 1$ is odd, without loss of generality assume $A$ is diagonal with decreasing eigenvalues $\lambda_1 \geq \cdots \geq \lambda_n$, let $\lambda = \lambda_{m+1}$ be the middle eigenvalue, and for each $1 \leq i \leq m$ find a convex combination $t_i \lambda_i + (1 - t_i) \lambda_{n + 1 - i} = \lambda$. Then the vectors $e_{m+1}$ and $\sqrt{t_i}e_i + \sqrt{1 - t_i}e_{n+1-i}$ for $1 \leq i \leq m$ together constitute an orthonormal basis of an $(m+1)$-dimensional subspace of $\mathbb{C}^n$ which is the range of an anticlique (\cite{weavery}, Lemma 2.3). If $n$ is even, any $\lambda$ lying between the two middle eigenvalues can be used, but it no longer contributes a separate eigenvector to the anticlique.

Operator systems of sufficiently low dimension are automatically triangle-free or strongly triangle-free.

\begin{prop}\label{lowdim}
Let $\mathcal{V}$ be an operator system. If ${\rm dim}(\mathcal{V}) \leq 8$ then $\mathcal{V}$ is triangle-free. If ${\rm dim}(\mathcal{V}) \leq 2$ then $\mathcal{V}$ is strongly triangle-free.
\end{prop}

\begin{proof}
It is clear that if ${\rm dim}(\mathcal{V}) \leq 8$ then we cannot have ${\rm dim}(P\mathcal{V}P) = 9$ for any rank 3 projection $P$. (The map $A \mapsto PAP$ is linear.) Thus $\mathcal{V}$ must be triangle-free. If ${\rm dim}(\mathcal{V}) \leq 2$ then $\mathcal{V}$ is the span of the identity and a single Hermitian matrix, and so the same is true of its compression to any three-dimensional subspace. Putting $n = 3$ in Proposition \ref{biganti} then yields that any such compression has a 2-anticlique. So $\mathcal{V}$ is strongly triangle-free.
\end{proof}

The interesting cases are therefore ${\rm dim}(\mathcal{V}) \geq 9$ (for triangle-free) and ${\rm dim}(\mathcal{V}) \geq 3$ (for strongly triangle-free). The latter will be completely characterized in Theorem \ref{maintheorem} below. As to the former, the following examples of nine-dimensional triangle-free operator systems in $M_4$ suggest that there will not be any simple characterization.

\begin{exmp}\label{m2ds}
Let $M_2 \oplus M_2$ be the operator system in $M_4$ consisting of all block diagonal matrices of the form $\left[\begin{matrix} A&0\cr 0&B\end{matrix}\right]$ with $A,B \in M_2$. This is an eight-dimensional operator system, so it is triangle-free by Proposition \ref{lowdim}. I claim that every one-dimensional extension of it is still triangle-free; thus we get a family of nine-dimensional triangle-free operator systems in $M_4$.

To prove the claim, let $P$ be a rank 3 projection in $M_4$; we will show that there is a nonzero matrix in $M_2 \oplus M_2$ whose compression to the range of $P$ is zero. This means that if $\mathcal{V}$ is any one-dimensional extension of $M_2 \oplus M_2$, so that ${\rm dim}(\mathcal{V}) = 9$, then ${\rm dim}(P\mathcal{V}P) \leq 8$ and so $P$ cannot be a 3-clique for $\mathcal{V}$.

Let $v$ be a unit vector with $Pv = 0$, so that the range of $P$ is the orthocomplement of $\mathbb{C}\cdot v$. Without loss of generality (replacing the original orthonormal basis $\{e_1, e_2, e_3, e_4\}$ of $\mathbb{C}^4$ with one of the form $\{e_1', e_2', e_3', e_4'\}$ such that ${\rm span}(e_1,e_2) = {\rm span}(e_1', e_2')$ and ${\rm span}(e_3, e_4) = {\rm span}(e_3', e_4')$) we can assume that $v$ has the form $v = (\alpha, 0, \beta, 0)$ for some $\alpha,\beta \in \mathbb{C}$. If $\beta = 0$ then the matrix $$\left[\begin{matrix}1&&&\cr &0&&\cr &&0&\cr &&&0\end{matrix}\right]$$ annihilates every vector in the range of $P$, which entails that its compression to the range of $P$ is zero. That proves the claim in this case. The case $\alpha = 0$ is similar.

Otherwise, if $\alpha$ and $\beta$ are both nonzero, then every vector $w$ that is orthogonal to $v$ has the form $w = (-a\overline{\beta}, b, a\overline{\alpha}, c)$ for some $a,b,c \in \mathbb{C}$ and the matrix $$\left[\begin{matrix}-\alpha/\overline{\beta}&&&\cr&0&&\cr &&\beta/\overline{\alpha}&\cr &&&0\end{matrix}\right]$$ takes every such vector to a scalar multiple of $v$. So again we have a nonzero matrix in $M_2 \oplus M_2$ whose compression to the range of $P$ is zero.
\end{exmp}

Example \ref{m2ds} provides a family of nine-dimensional triangle-free operator systems in $M_4$. Here is one more.

\begin{exmp}\label{skewm2ds}
Let $\mathcal{V} \subset M_4$ be the span of the identity matrix $I_4$ and all matrices of the form $$\left[\begin{matrix}0&A\cr B&0\end{matrix}\right]$$ with $A,B \in M_2$. We have to add in the identity matrix to get an operator system, but even without it we can make an argument similar to the one in Example \ref{m2ds} that there is something in the kernel of the compression to the range of any rank 3 projection. So ${\rm dim}(P\mathcal{V}P) \leq 8$ for any rank 3 projection $P$, and hence $\mathcal{V}$ is triangle-free.
\end{exmp}

The operator system in Example \ref{skewm2ds} is not unitarily equivalent to any of those in Example \ref{m2ds}. For instance, it contains no rank 1 projections. (Let $v = (a,b,c,d)$ be an arbitrary unit vector in $\mathbb{C}^4$ and compare the matrix entries of $vv^*$ to the form of the matrices in $\mathcal{V}$ from Example \ref{skewm2ds}.)

Finally, here is an example of a three-dimensional strongly triangle-free operator system.

\begin{exmp}\label{except}
Let $\mathcal{V} \subset M_4$ consist of all $4\times 4$ matrices of the form
$$\left[\begin{matrix}aI_2&bI_2\cr bI_2&cI_2\end{matrix}\right]$$ with $a,b,c \in \mathbb{C}$. I claim that every rank 3 projection dominates a 2-anticlique.

Let $A,B \in \mathcal{V}$ be the matrices with $a = 1$, $b = c = 0$ and $b = 1$, $a = c = 0$, respectively. Then $\mathcal{V} = {\rm span}(I_4, A, B)$. It will suffice to show that every rank 3 projection $Q$ dominates a rank 2 projection $P$ such that $PAP$ and $PBP$ are both scalar multiples of $P$. So let $Q$ be a rank 3 projection. By counting dimensions, the intersections of the range of $Q$ with ${\rm span}(e_1, e_2)$ and ${\rm span}(e_3, e_4)$ must both be nonzero. Without loss of generality $e_1,e_4 \in {\rm ran}(Q)$ and ${\rm ran}(Q)$ has an orthonormal basis of the form $\{e_1, v, e_4\}$ where $v = \sqrt{t}e_2 + \sqrt{1-t}e_3$ for some $t \in [0,1]$.

With respect to this basis, $QAQ$ has the form $$\left[\begin{matrix}1&0&0\cr 0&t&0\cr 0&0&0\end{matrix}\right]$$ and $QBQ$ has the form $$\left[\begin{matrix}0&\sqrt{1-t}&0\cr \sqrt{1-t}&0&\sqrt{t}\cr 0&\sqrt{t}&0\end{matrix}\right].$$ Taking $P$ to be the orthogonal projection onto the span of $v$ and $w = \sqrt{t}e_1 - \sqrt{1-t}e_4$, we get $PAP = tP$ and $PBP = 0$.
\end{exmp}

\subsection{Diagonal operator systems}

The {\it diagonal operator system} $\mathcal{D}_n$ consists of the diagonal matrices in $M_n$. It is an $n$-dimensional operator system that is spanned by the rank one projections $e_ie_i^*$.

Although the matrices in $\mathcal{D}_n$ are quite sparse (i.e., most entries are zero), it has no $2$-anticliques for any $n$ (\cite{weaverx}, Proposition 2.1). This follows from the fact that $\mathcal{D}_n$ is spanned by rank one matrices, which any putative 2-anticlique would have to compress to zero. Thus while $\mathcal{D}_n$ is trivially strongly triangle-free for $n = 1$ or $2$, it is not strongly triangle-free for any larger value of $n$. This simple observation has a surprisingly strong consequence!

\begin{prop}\label{noncom}
Let $\mathcal{V} \subseteq M_n$ be an operator system that contains commuting Hermitian matrices $A$ and $B$. Assume $I_n$, $A$, and $B$ are linearly independent. Then $\mathcal{V}$ is not strongly triangle-free.
\end{prop}

\begin{proof}
Fix an orthonormal basis $(v_i)$ that diagonalizes both $A$ and $B$ and say $Av_i = \lambda_i v_i$ and $Bv_i = \mu_i v_i$ for all $i$. Since $I_n$, $A$, and $B$ are linearly independent, without loss of generality we can assume that the vectors $(1, 1, 1)$, $(\lambda_1, \lambda_2, \lambda_3)$, and $(\mu_1, \mu_2, \mu_3)$ are linearly independent in $\mathbb{C}^3$, and therefore span it. Then the compression of $\mathcal{V}$ to ${\rm span}(v_1, v_2, v_3) \cong \mathbb{C}^3$ contains $\mathcal{D}_3$ and therefore has no 2-anticliques. So $\mathcal{V}$ is not strongly triangle-free.
\end{proof}

We also know that for sufficiently large $n$ the operator system $\mathcal{D}_n$ has $k$-cliques. The value given in Proposition 2.2 of \cite{weaverx} was $n \geq k^2 + k -1$. This is not far from being optimal, because $P\mathcal{D}_nP \cong M_k$ implies that $n = {\rm dim}(\mathcal{D}_n) \geq {\rm dim}(M_k) = k^2$. But in fact, by using techniques from frame theory the perfectly sharp bound $n = k^2$ can be achieved: $\mathcal{D}_{k^2}$ has a $k$-clique for every $k \geq 1$.

\begin{thm}
For any $k \geq 1$ the diagonal operator system $\mathcal{D}_{k^2}$ has a $k$-clique.
\end{thm}

\begin{proof}
Let $(v_r)_1^{k^2}$ be the $k^2$ vectors in $\mathbb{C}^k$ consisting of the $k$ standard basis vectors $e_i$, the $\frac{k^2 - k}{2}$ vectors $e_i + e_j$ ($i < j$), and the $\frac{k^2 - k}{2}$ vectors $e_i + \iota e_j$ ($i < j$). (Here I write $\iota = \sqrt{-1}$ to avoid confusion with the index $i$.) It is easy to see that the matrices $v_r v_r^*$ span $M_k$. Since ${\rm dim}(M_k) = k^2$, it follows that these matrices are linearly independent.

The sum $S = \sum v_r v_r^*$ is positive and greater than $I_k$, so it is invertible. Now for each $r$ define $w_r = S^{-1/2}v_r$. Thus $$\sum w_r w_r^* = S^{-1/2}\left(\sum v_r v_r^*\right) S^{-1/2} = S^{-1/2}SS^{-1/2} = I_k,$$
and linear independence of the matrices $v_r v_r^*$ implies linear independence of the matrices $w_r w_r^* = S^{-1/2}v_r v_r^* S^{-1/2}$.

We can now embed $\mathbb{C}^k$ in $\mathbb{C}^{k^2}$ by the map $T: v \mapsto \sum_{r = 1}^{k^2} \langle v, w_r\rangle e_r$. The adjoint map $T^*$ satisfies $T^*e_r = w_r$, so that $T^*Tv = \sum \langle v, w_r\rangle w_r = \sum w_r w_r^* v = v$ for all $v \in \mathbb{C}^k$, i.e., $T$ is an isometry. Meanwhile, $P = TT^*$ is the orthogonal projection of $\mathbb{C}^{k^2}$ onto ${\rm ran}(T) \cong \mathbb{C}^k$, and the matrices $$Pe_r e_r^*P = Tw_r w_r^* T^*$$ are linearly independent because the matrices $w_r w_r^*$ are linearly independent. Thus ${\rm dim}(P\mathcal{D}_{k^2}P) = k^2$ and so $P$ is a $k$-clique.
\end{proof}

(In the language of frame theory, we have found a Parseval frame $(w_r)_1^{k^2}$ in $\mathbb{C}^k$ for which the matrices $w_r w_r^*$ span $M_k$.)

\begin{cor}\label{d3}
$\mathcal{D}_n$ is strongly triangle-free only for $n \leq 2$ and triangle-free only for $n \leq 8$.
\end{cor}

\begin{cor}
Any abelian operator system whose dimension is at least 9 has a 3-clique, i.e., it is not triangle-free.
\end{cor}

(Since it compresses to $\mathcal{D}_9$, by the same reasoning as in Proposition \ref{noncom}.)

\subsection{Complete bipartite operator systems}

\begin{defn}\label{cbdef}
Let $m,n \geq 1$. The {\it complete bipartite operator system $\mathcal{K}_{m,n}$} is the linear subspace of $M_{m+n}$ consisting of all matrices of the form $$\left[\begin{matrix}aI_m&A\cr B&bI_n\end{matrix}\right]$$ with $a,b \in \mathbb{C}$, $A \in M_{m,n}$, and $B \in M_{n,m}$.
\end{defn}

Up to unitary equivalence, these are precisely the maximal two-colorable operator systems; see the comment following Proposition \ref{kkclique} below.

The operator systems $\mathcal{K}_{1,n}$ give us examples of strongly triangle-free operator systems of arbitrarily large dimension. (The dimension of $\mathcal{K}_{1,n}$ is $2n + 2$.) The fact that they are triangle-free was proven in Proposition 2.3 of \cite{weaverx}, but the argument given there actually shows that they are strongly triangle-free.

\begin{prop}\label{1ntf}
For any $n \geq 1$ the operator system $\mathcal{K}_{1,n}$ is strongly triangle-free.
\end{prop}

\begin{proof}
Let $P \in M_{n+1}$ be a rank 3 projection. Note that $\mathcal{K}_{1,n}$ is spanned by the operators $I_{n+1}$, $e_1e_1^*$, and $e_1e_i^*$ and $e_ie_1^*$ for $2 \leq i \leq n + 1$. Let $v = Pe_1$. If $v = 0$ then all of these operators except $I_{n+1}$ compress to zero on the range of $P$, so $P$ is a 3-anticlique. Otherwise, all of these operators except $I_{n+1}$ compress to operators of the form $vw^*$ or $wv^*$ for some $w = Pe_i \in {\rm ran}(P)$. Thus they all belong to $\mathcal{K}_{1,2}$ relative to an orthonormal basis whose first vector is $\frac{v}{\|v\|}$. So $P\mathcal{K}_{1,n}P \subseteq \mathcal{K}_{1,2}$ (in fact, they are equal), showing that $P\mathcal{K}_{1,n}P$ has a 2-anticlique. We have shown that every rank 3 projection dominates a 2-anticlique.
\end{proof}

But these are the only triangle-free complete bipartite operator systems. In marked contrast to the classical setting, if $m$ and $n$ are both at least $2$ then $\mathcal{K}_{m,n}$ has a 3-clique.

\begin{prop}\label{binot}
Let $m,n \geq 2$. Then $\mathcal{K}_{m,n}$ has a 3-clique.
\end{prop}

\begin{proof}
It suffices to consider the case $m = n = 2$. Let $P \in M_4$ be the rank 3 projection whose range is the orthocomplement of the span of the vector $(1,1,1,1) \in \mathbb{C}^4$. Let us determine when an arbitrary element $A = \left[\begin{matrix}aI_2&A_1\cr A_2&bI_2\end{matrix}\right]$ of $\mathcal{K}_{2,2}$ compresses to zero on the range of $P$. This happens if and only if it takes every vector in the range of $P$ to a scalar multiple of $(1,1,1,1)$.

The vector $v = (1,-1,0,0)$ is orthogonal to $(1,1,1,1)$ and hence belongs to the range of $P$. So $PAP = 0$ implies that $Av = (a, -a, \ast, \ast)$ is a scalar multiple of $(1,1,1,1)$, and this means, first, that $a = 0$, and second, that the two columns of $A_2$ are equal. Similar reasoning with the vector $(0,0,1,-1)$ shows that $b = 0$ and the two columns of $A_1$ are equal. Next, we use the fact that $A$ must take the vector $(1,0,-1,0)$ to a scalar multiple of $(1,1,1,1)$ to infer that $A_1 = \left[\begin{matrix}\alpha&\alpha\cr \alpha&\alpha\end{matrix}\right]$ and $A_2 = \left[\begin{matrix}-\alpha&-\alpha\cr -\alpha&-\alpha\end{matrix}\right]$ for some $\alpha \in \mathbb{C}$.

We have shown that the kernel of the compression map $A \mapsto PAP$ from $\mathcal{K}_{2,2}$ to $PM_4P \cong M_3$ is the span of the single matrix $$\left[\begin{matrix}&&1&1\cr &&1&1\cr -1&-1&&\cr -1&-1&&\end{matrix}\right].$$ (It is routine to check that this matrix does compress to zero on the range of $P$.) Thus the kernel is one-dimensional, and since ${\rm dim}(\mathcal{K}_{2,2}) = 10$ we conclude that ${\rm dim}(P\mathcal{K}_{2,2}P) = 9$. So $P$ is a 3-clique.
\end{proof}

\section{Triangle-free operator systems}

I noted in the introduction that every strongly triangle-free operator system is triangle-free. Actually, a little more is true.

\begin{prop}\label{plustwo}
Let $\mathcal{V} \subseteq \mathcal{W} \subseteq M_n$ be operator systems with ${\rm dim}(\mathcal{W}) \leq {\rm dim}(\mathcal{V}) + 2$. If $\mathcal{V}$ is strongly triangle-free, then $\mathcal{W}$ is triangle-free.
\end{prop}

\begin{proof}
Suppose $\mathcal{W}$ has a 3-clique $Q$ and let $P \leq Q$ be a rank 2 projection. Then $P\mathcal{W}P = PQ\mathcal{W}QP \cong M_2$ means that the dimension of the image of $\mathcal{W}$ under the map $A \mapsto PAP$ is 4, so the dimension of $P\mathcal{V}P$, the image of $\mathcal{V}$ under this map, is at least 2. So $P$ is not an anticlique for $\mathcal{V}$. Thus $Q$ does not dominate a 2-anticlique for $\mathcal{V}$. We have shown that if $\mathcal{W}$ is not triangle-free then $\mathcal{V}$ is not strongly triangle-free.
\end{proof}

Combining this result with Proposition \ref{1ntf} yields the following corollary.

\begin{cor}
For any $n \geq 1$, every two-dimensional extension of $\mathcal{K}_{1,n}$ is triangle-free. 
\end{cor}

This is sharp: according to Proposition 3.2 of \cite{weavery}, every three-dimensional extension of $\mathcal{K}_{1,n}$ {\it does} have a 3-clique.

I have two more general results about triangle-free operator systems, both taken from \cite{weavery}. First, let us define the {\it degree} of a unit vector $v \in \mathbb{C}^n$ relative to an operator system $\mathcal{V} \subseteq M_n$ to be the dimension of $\mathcal{V}v$. The {\it minimal degree} of $\mathcal{V}$ will then be the minimal value of this dimension, taken over all unit vectors $v$.

Classical triangle-free graphs can have arbitrarily large minimal degree, because every bipartite graph is triangle-free. But quantum bipartite graphs generally do have 3-cliques (Proposition \ref{binot}), and in fact we can give an absolute upper bound on the minimal degree of a quantum triangle-free operator system. The next result is Lemma 3.6 of \cite{weavery} with $k = 3$.

\begin{thm}\label{mind}
Every operator system with minimal degree at least 52,488 has a 3-clique.
\end{thm}

We can also ensure 3-cliques by bounding the dimension of the operator system. This is Theorem 3.7 of \cite{weavery} with $k = 3$.

\begin{thm}\label{linn}
Every operator system in $M_n$ whose dimension is at least $104,976n$ has a 3-clique.
\end{thm}

The constants in these results are large, but the fact that they exist at all is remarkable. The dimension of $M_n$ is $n^2$, so for $n = 104,976a$ Theorem \ref{linn} says that every operator system in $M_n$ that takes up at least $\frac{1}{a}$ of the available dimensions must have a 3-clique. Classically, one can have triangle-free bipartite graphs in which over half of the available edges are present ($n^2$ edges in the complete bipartite graph with $n + n$ vertices, out of $2n^2 - n$ possible edges).

Substantial improvements in the constants in Theorems \ref{mind} and \ref{linn} would be interesting to see.

\section{Coloring}

In ordinary graph theory, a $k$-coloring of a graph is given by partitioning the vertex set into at most $k$ anticliques, and the chromatic number of a graph is the smallest value of $k$ for which there exists a $k$-coloring. Multiple versions of these concepts have been proposed and studied in the quantum graph setting. I feel that the following definition is the most natural one.

\begin{defn}\label{colordef}
A {\it $k$-coloring} of an operator system $\mathcal{V} \subseteq M_n$ is a sequence of projections $P_1, \ldots, P_k \in M_n$ satisfying $P_1 + \cdots + P_k = I_n$, such that each $P_i$ is an anticlique for $\mathcal{V}$. The {\it chromatic number} of $\mathcal{V}$ is the smallest value of $k$ for which $\mathcal{V}$ can be $k$-colored.
\end{defn}

However, this definition has some counterintuitive properties. For instance, a $k$-colorable operator system can have a $k$-clique!

\begin{prop}\label{kkclique}
(\cite{weavery}, Proposition 3.4)
For any $k \geq 1$, the operator system $\mathcal{K}_{k,k}$ is $2$-colorable but it has a $k$-clique.
\end{prop}

Indeed, every complete bipartite operator system can be trivially 2-colored. Conversely, given a 2-coloring of an operator system $\mathcal{V}$ by projections $P$ and $Q$ of respective ranks $m$ and $n$, we will have $\mathcal{V} \subseteq \mathcal{K}_{m,n}$ relative to an orthonormal basis that diagonalizes both $P$ and $Q$. (This should explain a comment made after Definition \ref{cbdef}.) 

It is also easy to see that the diagonal operator system $\mathcal{D}_k$ cannot be colored with fewer than $k$ colors.

\begin{prop}
For any $k \geq 1$, the operator system $\mathcal{D}_k$ has a $k$-coloring but no $(k-1)$-coloring.
\end{prop}

This is simply because $\mathcal{D}_k$ has no $2$-anticliques (\cite{weaverx}, Proposition 2.1).

It is natural to ask whether triangle-free or strongly triangle-free operator systems have any special colorability properties. My main result in this direction is negative: there are two dimensional operator systems with arbitrarily high chromatic number. Since every such operator system is strongly triangle-free (Proposition \ref{lowdim}), this means that there is no universal bound on the number of colors needed to color a strongly triangle-free operator system. (However, it will follow from Theorem \ref{maintheorem} that every strongly triangle-free operator system with dimension greater than two is 2-colorable.)

Every two dimensional operator system in $M_n$ has the form ${\rm span}(I,A)$ for some Hermitian matrix $A \in M_n$, so let us work with the matrix $A$ directly.

\begin{defn}\label{singleanti}
A {\it $k$-anticlique} for a matrix $A \in M_n$ is a rank $k$ projection $P \in M_n$ for which $PAP$ is a scalar multiple of $P$. That is, it is a $k$-anticlique for the operator system ${\rm span}(I_n, A)$. A {\it $k$-coloring} of $A$ is a sequence of anticliques $P_1, \ldots, P_k \in M_n$ satisfying $P_1 + \cdots + P_k = I_n$. That is, it is a $k$-coloring of the operator system ${\rm span}(I_n,A)$. The {\it chromatic number} of $A$ is the smallest value of $k$ for which $A$ can be $k$-colored.
\end{defn}

We can completely classify the chromatic numbers of Hermitian matrices in $M_n$ with $n \leq 4$. The cases $n = 1$ or $2$ are trivial, as $A$ can have a 1-coloring if and only if it is a scalar multiple of the identity. For $n = 3$ or $4$ we can say the following.

\begin{prop}
Every Hermitian $A \in M_3$ has a 2-coloring.
\end{prop}

\begin{proof}
Use Proposition \ref{biganti} to find a $2$-anticlique $P$. Then $I_3 - P$ is a $1$-anticlique and it together with $P$ constitute a 2-coloring of $A$.
\end{proof}

\begin{prop}\label{inm4}
Every Hermitian $A \in M_4$ has a 3-coloring.
If $A$ has eigenvalues $\lambda_1 \geq \lambda_2 \geq \lambda_3 \geq \lambda_4$, then there is a $2$-coloring if and only if either $\lambda_2 = \lambda_3$ or $\lambda_1 + \lambda_4 = \lambda_2 + \lambda_3$.
\end{prop}

\begin{proof}
Let $\{v_1, v_2, v_3, v_4\}$ be an orthonormal basis of $\mathbb{C}^4$ such that each $v_i$ is an eigenvector with eigenvalue $\lambda_i$.

For the first part, use Proposition \ref{biganti} to find a 2-anticlique $P$. Then $Q = I_4 - P$ is a rank 2 projection and we can write $Q = Q_1 + Q_2$ where $Q_1$ and $Q_2$ are rank 1 projections. As every rank 1 projection is a 1-anticlique, the projections $P$, $Q_1$, and $Q_2$ constitute a 3-coloring of $A$.

Now suppose $\lambda_2 = \lambda_3$. Then there is a convex combination $t \lambda_1 + (1-t) \lambda_4 = \lambda_2 = \lambda_3$, and the vectors $\sqrt{t}v_1 + \sqrt{1-t}v_4$, $v_2$, $v_3$ are orthonormal and span the range of a 3-anticlique $P$. The projections $P$ and $I_4 - P$ then constitute a 2-coloring of $A$. Conversely, suppose $P$ is a 3-anticlique and $PAP = \mu P$. Then relative to an orthonormal basis of $\mathbb{C}^4$ whose last three vectors span the range of $P$, the matrix $B = A - \mu I_4$ has the form $$\left[\begin{matrix} a&b&c&d\cr \bar{b}&&&\cr \bar{c}&&0&\cr \bar{d}&&&\end{matrix}\right]$$ with $a$ real. By choosing the last three basis vectors more carefully, we can ensure that $c = d = 0$. If $b = 0$ as well then $A$ has an eigenvalue of multiplicity 3 or 4 and hence $\lambda_2 = \lambda_3$; otherwise, $B$ has the eigenvalue $0$ with multiplicity 2 and, since ${\rm det}\left[\begin{matrix}a&b\cr \overline{b}&0\end{matrix}\right] = -|b|^2 < 0$, also one strictly positive and one strictly negative eigenvalue. Thus again $\lambda_2 = \lambda_3$.

Finally, suppose $\lambda_1 + \lambda_4 = \lambda_2 + \lambda_3$. Then $$\left\{\frac{1}{\sqrt{2}}(v_1 + v_4), \frac{1}{\sqrt{2}}(v_2 + v_3), \frac{1}{\sqrt{2}}(v_1 - v_4), \frac{1}{\sqrt{2}}(v_2 - v_3)\right\}$$ is an orthonormal basis whose first two and last two vectors both span 2-anticliques. Thus $A$ has a 2-coloring. Conversely, suppose $A$ has a 2-coloring consisting of two rank 2 projections. Then relative to a suitable orthonormal basis $A$ can be put in the form $$\left[\begin{matrix}\mu_1I_2&B\cr B^*&\mu_2I_2\end{matrix}\right]$$ for some $\mu_1,\mu_2 \in \mathbb{R}$ and $B \in M_2$. Using the polar decomposition of $B$, we can find unitaries $U,V \in M_2$ such that $U^*BV$ is diagonal; then conjugating by the matrix $$\left[\begin{matrix}U&0\cr 0&V\end{matrix}\right]$$ puts $A$ in the form $$\left[\begin{matrix}\mu_1&0&b_1&0\cr 0&\mu_1&0&b_2\cr \overline{b}_1&0&\mu_2&0\cr 0&\overline{b}_2&0&\mu_2\end{matrix}\right],$$ or, reordering the basis, $$\left[\begin{matrix}\mu_1&b_1&&\cr \overline{b}_1&\mu_2&&\cr &&\mu_1&b_2\cr &&\overline{b}_2&\mu_2\end{matrix}\right].$$ Thus $A$ has two eigenvalues whose sum is $\mu_1 + \mu_2$ from the upper left corner, and two more eigenvalues whose sum is $\mu_1 + \mu_2$ from the lower right corner. We conclude that $\lambda_1 + \lambda_4 = \lambda_2 + \lambda_3$.
\end{proof}

Moving now to general values of $n$, Proposition \ref{biganti} entails the following result about coloring Hermitian matrices of arbitrary size.

\begin{thm}
Every Hermitian matrix in $M_n$, $n\geq 1$, has a $(\lfloor\log_2 n\rfloor + 1)$-coloring.
\end{thm}

\begin{proof}
We have $\lfloor \log_2 n\rfloor + 1 = k$ for $2^{k-1} \leq n < 2^k$. First consider the case $n = 2^k - 1 = 1 + 2 + \cdots + 2^{k-1}$. Here we apply Proposition \ref{biganti} to get a $2^{k-1}$-anticlique $P_1$, then compress to $I_n - P_1$ and get a $2^{k-2}$-anticlique $P_2$, and so on, ending with a 1-anticlique $P_k$. This yields a $k$-coloring of the original matrix.

For any $n < 2^k - 1$, follow the same procedure, applying Proposition \ref{biganti} repeatedly. At each stage the remaining dimension will be at most the remaining dimension at that stage from the case $n = 2^k - 1$, so it must terminate after at most $k$ steps, yielding a $k$-coloring of the original matrix.
\end{proof}

This bound is sharp.

\begin{thm}
For any $n \geq 1$ there is a Hermitian matrix in $M_n$ that has no $\lfloor\log_2 n\rfloor$-coloring.
\end{thm}

\begin{proof}
Let $A \in M_n$ be the diagonal matrix with diagonal entries $(n+1)^n$, $(n+1)^{n-1}$, $\ldots$, $n+1$. The eigenvectors of $A$ are the standard basis vectors $e_i$. I will show that if $n \geq 2^k$ then $A$ has no $k$-coloring.

Suppose $P_1$, $\ldots$, $P_k$ is a $k$-coloring of $A$, so that we have $P_iAP_i = \mu_i$ for some scalars $\mu_i$. Assume that $\mu_1 \geq \cdots \geq \mu_k$. For each $i$ let $d_i$ be the rank of $P_i$, and then define $D_i = d_1 + \cdots + d_i$. Thus $D_i$ is the rank of the projection $Q_i = P_1 + \cdots + P_i$.

Since $n\cdot \mu_1 \geq d_1\mu_1 + \cdots + d_k\mu_k = {\rm tr}(A) \geq (n + 1)^n$, we have $\mu_1 \geq \frac{1}{n}(n+1)^n > (n+1)^{n-1}$. This implies that $d_1 = 1$, because if the rank of $P_1$ was larger than 1 then its range would contain a unit vector orthogonal to $e_1$. But any unit vector $v$ orthogonal to $e_1$ satisfies $\langle Av,v\rangle \leq (n+1)^{n-1}$, whereas if it belonged to the range of $P_1$ then it would satisfy $\langle Av,v\rangle = \mu_1 > (n+1)^{n-1}$. We have shown that $\mu_1 > (n+1)^{n-1}$ and $D_1 = d_1 = 1$.

We now argue inductively that $\mu_{i+1} > (n+1)^{n - D_i - 1}$ and $d_{i+1} \leq 2^i$, so that $D_{i+1} < 2^{i+1}$. By the Schur-Horn theorem, the trace of $(I - Q_i)A(I - Q_i)$ is at least the sum of the $n - D_i$ smallest eigenvalues of $A$, the largest of which is $(n+1)^{n - D_i}$. Therefore $\mu_{i + 1} \geq \frac{1}{n}(n+1)^{n - D_i} > (n+1)^{n - D_i - 1}$. It follows that $P_{i+1}$ can have rank at most $D_i + 1$, because if its rank was larger than this then its range would contain a unit vector orthogonal to the span of $e_1$, $\ldots$, $e_{D_i + 1}$, but as in the base case, this would lead to a contradiction because any such vector $v$ would have to satisfy both $\langle Av,v\rangle \leq (n+1)^{n - D_i - 1}$ (since this is the norm of the compression of $A$ to the span of $e_{D_i + 2}$, $\ldots$, $e_n$) and $\langle Av,v\rangle = \mu_{i+1} > (n+1)^{n - D_i - 1}$. Thus we have shown that $\mu_{i+1} > (n+1)^{n - D_i - 1}$ and $d_{i+1} \leq D_i + 1$, so that inductively $d_{i+1} \leq 2^{i+1}$.

Taking $i = k$, this means that $D_k < 2^k \leq n$, contradicting the assumption that $Q_k = I_n$. We conclude that there is no $k$-coloring of $A$.
\end{proof}

\section{Strongly triangle-free operator systems}

We have seen that every two-dimensional operator system is strongly triangle-free (Proposition \ref{lowdim}) and that the complete bipartite operator systems $\mathcal{K}_{1,n}$ are strongly triangle-free (Proposition \ref{1ntf}). Thus any operator system contained in $\mathcal{K}_{1,n}$ is strongly triangle-free. It is also vacuously the case that every operator system in $M_2$ is strongly triangle-free, and we had one exceptional example of a strongly triangle-free operator system in Example \ref{except}.
In this section I will show that this list exhausts the strongly triangle-free operator systems.

The key tool is a characterization of the 2-anticliques of Hermitian operators in $M_3$.

\begin{lem}\label{lemma1}
Let $A \in M_3$ be Hermitian with eigenvalues $\lambda_1$, $\lambda_2$, $\lambda_3$ and corresponding unit eigenvectors $v_1$, $v_2$, $v_3$. If $\lambda_1 = \lambda_2 = \lambda_3$ then every rank 2 projection in $M_3$ is a 2-anticlique for $A$. If exactly two eigenvalues are equal then their eigenspace is the range of the unique 2-anticlique for $A$. If all three eigenvalues are distinct then, letting $\lambda_2$ be the middle eigenvalue, there is a unique $t \in (0,1)$ such that $t \lambda_1 + (1-t)\lambda_3 = \lambda_2$, and the 2-anticliques for $A$ are parametrized by the unit circle $|z| = 1$ as the projections onto the subspaces spanned by $v_2$ and $\sqrt{t} v_1 + z\sqrt{1-t}v_3$.
\end{lem}

\begin{proof}
If $A$ is a scalar multiple of the identity then every rank 2 projection is trivially an anticlique. Now let $P$ be a 2-anticlique for $A$, so that $PAP = \mu P$ for some $\mu \in \mathbb{C}$, and let $Q = I_3 - P$. Then $Q$ has rank 1 and ${\rm ran}(P)$ is two-dimensional, so there exists a nonzero vector $v \in {\rm ran}(P)$ with $QAv = 0$. So $$Av = (P + Q)Av = PAPv = \mu v.$$ This shows that $\mu$ is an eigenvalue of $A$. Without loss of generality say $\lambda_2 = \mu$.

The range of $P$ must therefore be spanned by $v_2$ and a unit vector of the form $w = a_1v_1 + a_3v_3$ where $|a_1|^2 + |a_3|^2 = 1$. But $$\lambda_2 = \langle Aw,w\rangle = |a_1|^2\lambda_1 + |a_3|^2\lambda_3,$$ i.e., $\lambda_2$ is a convex combination of $\lambda_1$ and $\lambda_3$. If two of the eigenvalues are equal, the only way this can happen is if one of $a_1$ and $a_3$ is zero and the range of $P$ is the eigenspace belonging to the repeated eigenvalue. If all three eigenvalues are distinct, then $\lambda_2$ must be the middle eigenvalue and $|a_1|^2$ and $|a_3|^2$ are uniquely determined. Letting $t = |a_1|^2$ and $1-t = |a_3|^2$, the only freedom in choosing the range of $P$ is the relative phase of $a_1$ and $a_3$, i.e., we can take $w = \sqrt{t}v_1 + z\sqrt{1-t}v_3$ for
arbitrary $|z| = 1$.
\end{proof}

\begin{lem}\label{lemma2}
Let $\mathcal{V} \subset M_3$ be a strongly triangle-free operator system and let $A, B \in \mathcal{V}$ be Hermitian. Suppose $A$ is diagonal, $$A = \left[\begin{matrix}\lambda_1&&\cr &\lambda_2&\cr &&\lambda_3\end{matrix}\right],$$ with $\lambda_1 > \lambda_2 > \lambda_3$. Then $$(\lambda_2 - \lambda_3)|b_{12}|^2 = (\lambda_1 - \lambda_2)|b_{32}|^2$$ where $B = (b_{ij})$.
\end{lem}

\begin{proof}
Since the eigenvalues of $A$ are distinct, we are in the last case of Lemma \ref{lemma1}. Here $t = \frac{\lambda_2 - \lambda_3}{\lambda_1 - \lambda_3}$ and $1-t = \frac{\lambda_1 - \lambda_2}{\lambda_1 - \lambda_3}$. Since $\mathcal{V}$ is strongly triangle-free, it has a 2-anticlique, which by Lemma \ref{lemma1} has range spanned by $e_2$ and $v = \sqrt{t}e_1 + z\sqrt{1-t}e_3$ for some $|z| = 1$. Thus $\langle B e_2, v\rangle = 0$, i.e., $$\sqrt{t}b_{12} + \overline{z}\sqrt{1-t}b_{32} = 0.$$ This implies that $\sqrt{t}| b_{12}| = \sqrt{1-t}|b_{32}|$, or, squaring, $(\lambda_2 - \lambda_3)|b_{12}|^2 = (\lambda_1 - \lambda_2)|b_{32}|^2$.
\end{proof}

\begin{lem}\label{lemma3}
Let $\mathcal{V} \subset M_4$ be a strongly triangle-free operator system and let $A, B \in \mathcal{V}$. Suppose $A$ is diagonal, $$A = \left[\begin{matrix}\lambda_1&&&\cr &\lambda_2&&\cr &&\lambda_3&\cr &&&\lambda_4\end{matrix}\right],$$ with $\lambda_1 > \lambda_2 > \lambda_3$, and $b_{12} \neq 0$ where $B = (b_{ij})$. Then $\lambda_2 = \lambda_4$.
\end{lem}

\begin{proof}
For sufficiently small $\epsilon > 0$, every $t \in [0,\epsilon]$ has the property that $\lambda_2 > \langle Av, v\rangle$ for $v = \sqrt{t}e_4 + z\sqrt{1-t}e_3$ (with $|z| = 1$, but otherwise arbitrary).
Compress $\mathcal{V}$ onto the three-dimensional subspace of $\mathbb{C}^4$ spanned by $e_1$, $e_2$, and $v$. Then $A$ compresses to a diagonal matrix with diagonal entries $\lambda_1$, $\lambda_2$, and $\langle Av,v\rangle = t\lambda_4 + (1-t)\lambda_3$ and Lemma \ref{lemma2} yields  \begin{equation}\big(\lambda_2 - (t\lambda_4 + (1-t)\lambda_3)\big)|b_{12}|^2 = (\lambda_1 - \lambda_2)|\sqrt{t}b_{42} + z\sqrt{1-t}b_{32}|^2\tag{$\ast$}\end{equation} for all $t \in [0,\epsilon]$ and $|z| = 1$. Squaring out the right side yields $$|\sqrt{t}b_{42} + z\sqrt{1-t}b_{32}|^2 = t|b_{42}|^2 + 2\sqrt{t(1-t)}\cdot{\rm Re}(\overline{z} b_{42}\overline{b}_{32}) + (1-t)|b_{32}|^2.$$  Now every term in ($\ast$) besides $2\sqrt{t(1-t)}\cdot{\rm Re}(\overline{z} b_{42}\overline{b}_{32})$ is differentiable at $t = 0$, so we must have ${\rm Re}(\overline{z} b_{42}\overline{b}_{32}) = 0$; since $|z| = 1$ was arbitrary this means that either $b_{32} = 0$ or $b_{42} = 0$. If $b_{32} = 0$ then putting $t = 0$ in ($\ast$) yields $(\lambda_2 - \lambda_3)|b_{12}|^2 = 0$, but this is impossible since $b_{12} \neq 0$ and $\lambda_2 > \lambda_3$. So $b_{42} = 0$, and then ($\ast$) simplifies to $$((\lambda_2 - \lambda_3) - (\lambda_4 - \lambda_3)t)|b_{12}|^2 = (\lambda_1 - \lambda_2)(1 - t)|b_{32}|^2,$$ but since $b_{12} \neq 0$ and the right side is a constant times $1 - t$, that forces $\lambda_2 - \lambda_3 = \lambda_4 - \lambda_3$, which implies that $\lambda_2 = \lambda_4$.
\end{proof}

\begin{lem}\label{lemma4}
Let $\mathcal{V} \subset M_4$ be a triangle-free operator system and let $A, B \in \mathcal{V}$. Suppose $A$ is diagonal, $$A = \left[\begin{matrix}\lambda_1&&&\cr &\lambda_2&&\cr &&\lambda_3&\cr &&&\lambda_4\end{matrix}\right],$$ with $\lambda_1 \geq \lambda_2 \geq \lambda_3 > \lambda_4$, $\lambda_1 \neq \lambda_3$, and $b_{14} \neq 0$ where $B = (b_{ij})$. Then $\lambda_2 = \lambda_3$.
\end{lem}

\begin{proof}
Let $|z| = 1$ and define $v^\pm = \sqrt{t}e_1 \pm z\sqrt{1-t}e_4$ where $t = \frac{\lambda_3 - \lambda_4}{\lambda_1 - \lambda_4} \in (0,1)$. Then $\langle Av^\pm,v^\pm\rangle = \lambda_3$, so if $\lambda_2 \neq \lambda_3$ then the compression of $A$ to ${\rm span}(e_2, e_3, v^\pm)$ has the eigenvalue $\lambda_3$ with multiplicity exactly 2. By Lemma \ref{lemma1}, this compression has a unique 2-anticlique, with range ${\rm span}(e_3,v^\pm)$, and therefore this must be a 2-anticlique for $B$ as well. Thus $$\langle Bv^+, v^+\rangle = \langle Be_3,e_3\rangle = \langle Bv^-,v^-\rangle,$$ but evaluating the left and right sides of this yields $$2\sqrt{t(1-t)}{\rm Re}\, \overline{z}b_{14} = -2\sqrt{t(1-t)}{\rm Re}\, \overline{z}b_{14},$$ contradicting the assumption that $b_{14} \neq 0$ (since $z$ was arbitrary with $|z| =1$, and $t \in (0,1)$). We conclude that $\lambda_2 = \lambda_3$.
\end{proof}

\begin{lem}\label{lemma5}
Let $\mathcal{V} \subset M_n$ be a strongly triangle-free operator system, $n \geq 3$, and suppose $e_1e_1^* \in \mathcal{V}$. Then $\mathcal{V} \subseteq \mathcal{K}_{1,n-1}$.
\end{lem}

\begin{proof}
For any $2\leq i < j\leq n$, the compression of $e_1e_1^*$ to ${\rm span}(e_1, e_i, e_j)$ has a unique (Lemma \ref{lemma1}) 2-anticlique with range ${\rm span}(e_i,e_j)$, so ${\rm span}(e_i,e_j)$ must be an anticlique for $\mathcal{V}$. This easily implies that ${\rm span}(e_2, \ldots, e_n)$ is an anticlique for $\mathcal{V}$, so $\mathcal{V} \subseteq \mathcal{K}_{1,n-1}$.
\end{proof}

\begin{lem}\label{lemma6}
Let $\mathcal{V} \subset M_n$ be a strongly triangle-free operator system with ${\rm dim}(\mathcal{V}) \geq 3$ and let $A \in \mathcal{V}$ be Hermitian. Suppose $A$ has exactly two eigenvalues, both with multiplicity $\geq 2$. Then $n = 4$ and $\mathcal{V}$ is the operator system from Example \ref{except}, relative to some choice of basis of $\mathbb{C}^4$.
\end{lem}

\begin{proof}
First take the case where $n = 4$ and both eigenvalues of $A$ have multiplicity exactly 2. By choosing a basis that diagonalizes $A$, subtracting a multiple of $I_4$, and scaling, we can assume $A$ has the form $$\left[\begin{matrix}1&&&\cr &1&&\cr &&0&\cr &&&0\end{matrix}\right].$$

Since ${\rm dim}(\mathcal{V}) \geq 3$, according to Proposition \ref{noncom} there is a Hermitian matrix $B \in \mathcal{V}$ that does not commute with $A$. But the compression of $A$ to ${\rm span}(e_1,e_2,e_3)$ has three eigenvalues, exactly two of which are equal, so by Lemma \ref{lemma1} there is a unique anticlique for $A$, namely the projection onto ${\rm span}(e_1,e_2)$, whose range is contained in this span. This must therefore be an anticlique for $\mathcal{V}$. By similar reasoning, the projection onto ${\rm span}(e_3,e_4)$ is also an anticlique for $\mathcal{V}$. In particular, these are both 2-anticliques for $B$. Thus, by subtracting a linear combination of $I_4$ and $A$ from $B$, we can assume that its compressions to ${\rm span}(e_1,e_2)$ and ${\rm span}(e_3,e_4)$ are both zero.

As in the proof of Proposition \ref{inm4} we can now find unitaries $U,V \in M_2$ such that conjugating by the matrix $$\left[\begin{matrix}U&0\cr 0&V\end{matrix}\right]$$ puts $B$ in the form $$\left[\begin{matrix}&&a&\cr &&&b\cr a&&&\cr &b&&\end{matrix}\right]$$ with $a,b \geq 0$. Without loss of generality $a \neq 0$, and by scaling we can assume $a = 1$. Now let $Q$ be the projection onto ${\rm span}(e_1, \frac{1}{\sqrt{2}}(e_2 + e_3), e_4)$, so that $$QAQ = \left[\begin{matrix}1&&\cr &\frac{1}{2}&\cr &&0\end{matrix}\right]$$ and $$QBQ = \left[\begin{matrix}0&\frac{1}{\sqrt{2}}&0\cr \frac{1}{\sqrt{2}}&0&\frac{1}{\sqrt{2}}b\cr 0&\frac{1}{\sqrt{2}}b&0\end{matrix}\right].$$ Lemma \ref{lemma2} now shows that $b = 1$. Thus $\mathcal{V}$ contains the operator system from Example \ref{except}. But the argument about $B$ could now be repeated for any Hermitian $C \in \mathcal{V}$ that does not commute with $A$ to show that after subtracting a linear combination of $I_4$ and $A$ from $C$, relative to the basis which we have been using for $A$ and $B$, the matrix for $C$ must contain a scalar multiple of a $2\times 2$ unitary in its upper right corner. This will be true of any linear combination of it with $B$, but the only $2\times 2$ matrices all of whose linear combinations with $I_2$ are scalar multiples of a unitary are scalar multiples of $I_2$. Thus $C$ is, after having subtracted something in ${\rm span}(I_4,A)$, a scalar multiple of $B$. So in fact $\mathcal{V}$ must equal the operator system from example \ref{except}.

That takes care of the case where both eigenvalues of $A$ have multiplicity exactly 2. If either of them has strictly larger multiplicity, i.e., if at least one of the eigenspaces $E_1$ and $E_2$ has dimension at least 3, then we can assume the eigenvalues of $A$ are $0$ and $1$ and find a Hermitian $B \in \mathcal{V}$ which does not commute with $A$ and two-dimensional subspaces $E_1'$ of $E_1$ and $E_2'$ of $E_2$ such that $B$ does not take $E_2'$ unitarily onto $E_1'$. The preceding analysis of the $4\times 4$ case now shows that the compression of $\mathcal{V}$ to $E_1' + E_2'$ is not strongly triangle-free, and therefore $\mathcal{V}$ could not have been strongly triangle-free either. We conclude that both eigenvalues of $A$ must have had multiplicity exactly 2.
\end{proof}

\begin{lem}\label{lemma7}
Let $\mathcal{V} \subset M_n$ be a strongly triangle-free operator system with ${\rm dim}(\mathcal{V}) \geq 3$. Then every Hermitian $A \in \mathcal{V}$ has at most three eigenvalues. If it has exactly three eigenvalues, then only the middle eigenvalue can have multiplicity $\geq 2$.
\end{lem}

\begin{proof}
Let $A \in \mathcal{V}$ be Hermitian and suppose $A$ has at least four distinct eigenvalues. Since ${\rm dim}(\mathcal{V}) \geq 3$, by Proposition \ref{noncom} there exists a Hermitian matrix $B \in \mathcal{V}$ which does not commute with $A$. There must then exist eigenvectors $v$ and $w$ of $A$ belonging to distinct eigenvalues which satisfy $\langle Bv,w\rangle \neq 0$.

By compressing onto the span of exactly four eigenspaces of $A$, we can assume that $A$ has exactly four eigenvalues $\lambda_1 > \lambda_2 > \lambda_3 > \lambda_4$. A moment's thought shows that Lemma \ref{lemma3} (possibly applied to $-A$ in place of $A$) forbids either $v$ or $w$ from belonging to the $\lambda_2$ or $\lambda_3$ eigenspaces. But this puts us in the situation of Lemma \ref{lemma4}, which contradicts the distinctness of $\lambda_2$ and $\lambda_3$. We conclude that $A$ could not have had more than three distinct eigenvalues.

There still exist $v$ and $w$ belonging to distinct eigenvalues of $A$ such that $\langle Bv,w\rangle \neq 0$. If $A$ has exactly three eigenvalues, $\lambda_1 > \lambda_2 > \lambda_3$, and either $v$ or $w$ belongs to the $\lambda_2$ eigenspace, then Lemma \ref{lemma3} entails that $\lambda_1$ and $\lambda_3$ have no multiplicity. If they belong to the $\lambda_1$ and $\lambda_3$ eigenspaces, then Lemma \ref{lemma4} yields the same conclusion. So only the middle eigenvalue can have multiplicity $\geq 2$.
\end{proof}

\begin{thm}\label{maintheorem}
Let $\mathcal{V} \subseteq M_n$ be an operator system. Then $\mathcal{V}$ is strongly triangle-free if and only if one of the following holds:
\begin{enumerate}
\item $n = 2$

\item ${\rm dim}(\mathcal{V}) = 2$

\item $n = 4$ and $\mathcal{V}$ is the operator system from Example \ref{except}, relative to some choice of basis of $\mathbb{C}^4$

\item $\mathcal{V} \subseteq \mathcal{K}_{1,n-1}$, relative to some choice of basis of $\mathbb{C}^n$.
\end{enumerate}
\end{thm}

\begin{proof}
The condition $n = 2$ vacuously entails that $\mathcal{V}$ is strongly triangle-free. That the other conditions yield the same conclusion was shown in Proposition \ref{lowdim}, Example \ref{except}, and Proposition \ref{1ntf}.

Conversely, suppose $n \geq 3$ and ${\rm dim}(\mathcal{V}) \geq 3$. (If $n = 1$ then $\mathcal{V} = \mathcal{K}_{1,0}$, and if ${\rm dim}(\mathcal{V}) = 1$ then $\mathcal{V} = \mathbb{C}\cdot I_n \subseteq \mathcal{K}_{1,n-1}$.) We can assume $n \geq 4$ because if $n = 3$ the existence of a 2-anticlique immediately shows that, after a change of basis, we have $\mathcal{V} \subseteq \mathcal{K}_{1,2}$.

We know from Lemma \ref{lemma7} that every Hermitian matrix in $\mathcal{V}$ has at most three distinct eigenvalues. Consider first the case where some Hermitian matrix in $\mathcal{V}$ has exactly two eigenvalues. If both have multiplicity $\geq 2$ then we are done by Lemma \ref{lemma6}; if one has multiplicity $1$ then we can subtract a multiple of $I_n$ so that the eigenvalue with multiplicity $\geq 2$ is 0, and then infer that $\mathcal{V}$ contains a rank 1 projection. After a change of basis, we can conclude that $\mathcal{V} \subseteq \mathcal{K}_{1,n-1}$ by Lemma \ref{lemma5}.

Otherwise, every non-scalar Hermitian matrix in $\mathcal{V}$ has exactly three eigenvalues, only the middle of which has multiplicity $\geq 2$ (Lemma \ref{lemma7}). Let $\mathcal{V}_0$ be the set of Hermitian matrices in $\mathcal{V}$ whose middle eigenvalue is zero, together with the zero matrix. Every Hermitian matrix in $\mathcal{V}$ differs additively from something in $\mathcal{V}_0$ by a scalar multiple of the identity. I claim that $\mathcal{V}_0$ is a real linear space. To see this, let $A, B \in \mathcal{V}_0$ and let $P$ be a 2-anticlique for $\mathcal{V}$. Now if $A$ is nonzero, it has exactly one strictly positive eigenvalue and one strictly negative eigenvalue, both without multiplicity, so by the Cauchy interlacing theorem $PAP$ cannot have two strictly positive or two strictly negative eigenvalues. Thus since $PAP$ has only one eigenvalue ($P$ is an anticlique), that eigenvalue must be $0$. The same reasoning applies to $B$, so we conclude that $PAP = PBP = 0$. Thus $P(A+B)P = 0$. We have shown that $A + B$ compresses to zero on any 2-anticlique for $\mathcal{V}$. But $A+B$, if nonzero, has a middle eigenvalue with multiplicity; the projection onto the span of two eigenvectors for this middle eigenvalue plus one eigenvector for the top eigenvalue then dominates only one 2-anticlique for $A + B$ (Lemma \ref{lemma1}), so this must be an anticlique for $\mathcal{V}$ as well, and this shows that the middle eigenvalue of $A + B$ must be $0$. This proves the claim.

Let $A,B \in \mathcal{V}_0$ with $AB \neq BA$. Now $A$ has one strictly positive and one strictly negative eigenvalue, both without multiplicity, so we can write $A = v_1v_1^* - v_2v_2^*$ for some $v_1,v_2 \in \mathbb{C}$; similarly we can write $B = w_1w_1^* - w_2w_2^*$ for some $w_1, w_2 \in \mathbb{C}$. The four vectors $v_1$, $v_2$, $w_1$, $w_2$ cannot be linearly independent. If they were we could find $u \in \mathbb{C}$ with $u \perp v_1, v_2, w_1$ but $u \not\perp w_2$; then $(A + B)u = -\langle u, w_2\rangle w_2$, showing that $w_2 \in {\rm ran}(A + B)$, and similarly $v_1, v_2, w_1 \in {\rm ran}(A + B)$, contradicting the fact that everything nonzero in $\mathcal{V}_0$ has rank 2.

So there is a three-dimensional subspace $E$ of $\mathbb{C}^n$ which contains the ranges of both $A$ and $B$, and the projection $Q$ onto $E$ dominates a 2-anticlique for $\mathcal{V}$. Choose a basis $\{v, v', v''\}$ for $E$ so that the projection onto $E_0 = {\rm span}(v',v'')$ is an anticlique for $\mathcal{V}$. We can assume that ${\rm ran}(A) = {\rm span}(v,v')$. (We have $v \in {\rm ran}(A)$ since $A$ compresses to something in $\mathcal{K}_{1,2}$ with rank 2.) Let $F$ be the orthocomplement of $\mathbb{C}\cdot v$ in $\mathbb{C}^n$. To complete the proof, we will show that everything in $\mathcal{V}_0$ compresses to zero on $F$.

Multiplying $A$ by a nonzero scalar, we can assume that the matrix for $QAQ$ has the form $$\left[\begin{matrix}a&1&0\cr 1&&\cr 0&&\end{matrix}\right],$$ and subtracting a multiple of $A$ from $B$ and scaling, we can assume that the matrix for $QBQ$ has the form $$\left[\begin{matrix}b&0&1\cr 0&&\cr 1&&\end{matrix}\right].$$ (If the range of $B$ equalled the range of $A$ then we could subtract a multiple of $A$ from $B$ and scale to get the operator $vv^*$, but we are working in the case that every non-scalar Hermitian matrix in $\mathcal{V}$ has exactly three eigenvalues.)

Now we know that everything in $\mathcal{V}_0$ compresses to zero on $E_0$. Next, any two-dimensional subspace of $F \ominus E_0$ is contained in the zero eigenspace for $A$, so it is the range of the unique 2-anticlique for $A$ supported on the three-dimensional space spanned by it plus one eigenvector of $A$ belonging to a nonzero eigenvalue (Lemma \ref{lemma1}); thus everything in $\mathcal{V}_0$ compresses to zero on it. Finally, we can make the same argument (relative to $A$, since $v'' \in {\rm ker}(A)$) for the two-dimensional subspace spanned by $v''$ and any nonzero vector in $F \ominus E_0$, or (relative to $B$, since $v' \in {\rm ker}(B)$) for the two-dimensional subspace spanned by $v'$ and any nonzero vector in $F\ominus E_0$. All together, this shows that everything in $\mathcal{V}_0$ compresses to zero on $F$.
\end{proof}

\bibliographystyle{amsplain}

\end{document}